\documentclass[11pt,3p]{elsarticle}

\usepackage{amsfonts}
\usepackage{amsthm}
\usepackage{amsmath}
\usepackage{times}
\usepackage{lineno}
\usepackage{multirow,array}

\newtheorem{theorem}{Theorem}
\newtheorem{proposition}[theorem]{Proposition}
\newtheorem{lemma}[theorem]{Lemma}

\newtheorem{definition}{Definition}

\newtheorem{problem}{Problem}
\newtheorem{remark}{Remark}

\renewcommand{\epsilon}{\varepsilon}

\DeclareMathOperator{\pap}{\textbf{p}}

\renewcommand{\epsilon}{\varepsilon}

\def\A{\mathbb{A}}

\begin{document}

\sloppy


\begin{frontmatter}

\title{Abelian AntiPowers in Infinite Words}

\author[label1]{Gabriele Fici\corref{cor1}}
\ead{gabriele.fici@unipa.it}

\author[label2]{Mickael Postic}
\ead{postic@math.univ-lyon1.fr}

\author[label3]{Manuel Silva}
\ead{mnas@fct.unl.pt}

\address[label1]{Dipartimento di Matematica e Informatica, Universit\`a di Palermo, Palermo, Italy}

\address[label2]{Institut Camille Jordan, Universit\'e Claude Bernard Lyon 1,  Lyon, France}

\address[label3]{Faculdade de Ci\^encias e Tecnologia, Universidade Nova de Lisboa, Lisbon, Portugal}

\cortext[cor1]{Corresponding author.}

\journal{Advances in Applied Mathematics}

\begin{abstract}
An abelian antipower of order $k$ (or simply an abelian $k$-antipower) is a concatenation of $k$ consecutive words of the same length having  pairwise distinct Parikh vectors. This definition  generalizes to the abelian setting the notion of a $k$-antipower, as introduced in [G. Fici et al., {\em antipowers in infinite words}, J.~Comb.~Theory,~Ser.~A, 2018], that is a concatenation of $k$ pairwise distinct words of the same length. 
We aim to study whether a word contains abelian $k$-antipowers for arbitrarily large $k$. \v{S}.~Holub proved that all paperfolding words contain  abelian powers of every order [{\em Abelian powers in paperfolding words}. J.~Comb.~Theory, Ser.~A, 2013].  We show that they also contain  abelian antipowers of every order. 
\end{abstract}

\begin{keyword}
Abelian antipower; $k$-antipower; abelian complexity; paperfolding word; Sierpi\`nski word.
\end{keyword}

\end{frontmatter}

\section{Introduction}

Many of the classical definitions in combinatorics on words (e.g., period, power, factor complexity, etc.) have a counterpart in the abelian setting, though they may not enjoy the same properties.

Recall that the Parikh vector $P(w)$ of a word $w$ over a finite ordered alphabet $\A=\{a_1,a_2,\ldots,a_{|\A|}\}$ is the vector whose $i$-th component is equal to the number of occurrences of the letter $a_i$ in $w$, $1\leq i \leq |\A|$. For example, the Parikh vector of $w=abbca$ over $\A=\{a,b,c\}$ is $P(w)=(2,2,1)$. This notion is at the basis of the abelian combinatorics on words, where two words are considered equivalent if and only if they have the same Parikh vector. 

For example, the classical notion of factor complexity (the function that counts the number of distinct factors of length $n$ of a word, for every $n$) can be generalized by considering the so-called abelian factor complexity (or abelian complexity for short), that is the function that counts the number of distinct Parikh vectors of factors of length $n$, for every $n$.

Morse and Hedlund~\cite{MoHe} proved that an infinite word is aperiodic if and only if its factor complexity is unbounded. This characterization does not have an analogue in the case of the abelian complexity, as there exist aperiodic words with bounded abelian complexity. For example, the well-known Thue-Morse word has abelian complexity bounded by $3$, yet it is aperiodic.

Richomme et al.~\cite{minsub}  proved that if a word has bounded abelian complexity, then it contains abelian powers of every order --- an abelian power of order $k$ is a concatenation of $k$ words having the same Parikh vector. 
However, this is not a characterization of words with bounded abelian complexity. 
Indeed, \v{S}t\v{e}p\'{a}n Holub~\cite{Holub} proved that all paperfolding words contain abelian powers of every order, and paperfolding words have unbounded abelian complexity (a property that by the way follows from the main result of this paper). The class of paperfolding words therefore constitutes an interesting example, as they are uniformly recurrent (every factor appears infinitely often and with bounded gaps) aperiodic words with linear factor complexity.

In a recent paper \cite{icalp}, the first and the third author, together with Antonio Restivo and Luca Zamboni, introduced the notion of an antipower. An {\em antipower of order $k$}, or simply a \emph{$k$-antipower}, is a concatenation of $k$ consecutive pairwise distinct words of the same length.  E.g., $aabaaabbbaba$ is a $4$--antipower.

In \cite{icalp}, it is proved that the existence of powers of every order or antipowers of every order is an unavoidable regularity for infinite words:

\begin{theorem}\label{unav}\cite{icalp}
Every infinite word contains powers of every order or antipowers of every order.
\end{theorem}

Note that in the previous statement there is no hypothesis on the alphabet size. 

Actually, in~\cite{icalp} a stronger result is proved (of which we omit the statement here for the sake of simplicity) from which it follows that every aperiodic uniformly recurrent word must contain antipowers of every order.

In this paper, we extend the notion of an antipower to the abelian setting. 

\begin{definition}
 An {\em abelian antipower of order $k$}, or simply an \emph{abelian $k$-antipower}, is a concatenation of $k$ consecutive words of the same length having pairwise distinct Parikh vectors.  
\end{definition}

For example, $aabaaabbbabb$ is an abelian $4$--antipower. Notice that an abelian $k$-antipower is a $k$-antipower but the converse does not necessarily hold (which is dual to the fact that a $k$--power is an abelian $k$--power but the converse does not necessarily hold).

We think that an analogue of Theorem \ref{unav} may still hold in the case of abelian antipowers, but unfortunately the proof of Theorem~\ref{unav} does not generalize to the abelian setting.

\begin{problem}\label{conj:ab}
Does every infinite word contain abelian powers of every order or abelian antipowers of every order?
\end{problem}

Clearly, if a word has bounded abelian complexity, then it cannot contain abelian antipowers of every order. However, a word can avoid large abelian antipowers even if its abelian complexity is unbounded. Indeed, in \cite{icalp}, an example is shown of an aperiodic recurrent  word  avoiding  $6$-antipowers (and therefore avoiding abelian $6$-antipowers), and from the construction it can be easily verified that the abelian complexity of this word is unbounded. 

A similar situation can be illustrated with the well-known Sierpi\`nski word. Recall that the Sierpi\`nski word (also known as Cantor word) $s$ is the fixed point starting with $a$ of the substitution
\begin{align*}
\sigma:\ & a \rightarrow aba  \\
& b \rightarrow bbb
\end{align*}
so that the word $s$ begins as follows:
\[ababbbababbbbbbbbbababbbabab^{27}a\cdots\]

Therefore, $s$ can be obtained as the limit, for $n\to \infty$, of the sequence of words $(s_n)_{n\geq 0}$ defined by: $s_0=a$,  $s_{n+1}=s_nb^{3^n}s_n$ for $n\geq 1$. Notice that for every $n$ one has $|s_n|=3^n$.

We show that the abelian complexity of $s$ is unbounded.

\begin{theorem}\label{thm:S}
The Sierpi\`nski word $s$ does not contain $11$--antipowers, hence it does not contain abelian $11$--antipowers.
\end{theorem}

An infinite word can contain both abelian powers of every order and abelian antipowers of every order. This is the case, for example, of any word with full factor complexity. However, finding a class of uniformly recurrent words with linear factor complexity satisfying this property seems a more difficult task. Indeed, most of the well-known examples (Thue-Morse, Sturmian words, etc.) have bounded abelian complexity, hence they cannot contain abelian antipowers of every order --- whereas, by the aforementioned result of Richomme et al.~\cite{minsub}, they contain abelian powers of every order.
Building upon the framework that \v{S}t\v{e}p\'{a}n Holub developed to prove that all paperfolding words contain abelian powers of every order \cite{Holub}, we prove in the next section that all paperfolding words contain also abelian antipowers of every order.

\section{Sierpi\`nski Word}

Blanchet-Sadri, Fox and Rampersad \cite{BFR} characterized the asymptotic behavior of the abelian complexity of words that are fixed points of a morphism.
In the following proposition, we give the precise bounds of the abelian complexity of the Sierpi\`nski word.

\begin{proposition}
The abelian complexity $a(n)$ of the Sierpi\`nski word verifies $a(n)=\Theta(n^{\log_3 2 })$.
\end{proposition}

\begin{proof}
The Sierpi\`nski word $s$ is prefix normal with respect to the letter $a$ (see~\cite{FL11,PN17} for the definition of prefix normal word), that is, for each length $n$, no factor of $s$ of length $n$ contains more occurrences of the letter $a$ than the prefix of length $n$. Since $s$ contains arbitrarily long blocks of $b$s, the number of distinct Parikh vectors of factors of $s$ of a given length $n$ is given by $1$ plus the number of $a$s in the prefix of length $n$. It is easy to see that the values of $n$ for which the proportion of $a$'s is maximal in a prefix of length $n$ are of the form $n=3^k$, while those for which the proportion of $a$'s is minimal are of the form $n=2\cdot 3^k$, and in both cases the prefix of length $n$ contains $2^k$ $a$s. With a standard algebraic manipulation, this gives
\[n^{\log_3 2 }/2^{\log_3 2 }\leq a(n)\leq n^{\log_3 2 }.\]
\end{proof}

{\bf Proof of Theorem \ref{thm:S}}.
Suppose that $s$ contains an $11$--antipower $u=u_1u_2\cdots u_{11}$, of length $11m$.
Let us then consider the first occurrence of $u$ in $s$. Let $n$ be the smallest integer such that $u$ occurs in $s_{n+1}b^{3^{n+1}}$  but not in $s_{n}b^{3^{n}}$. 

Let us first suppose that no $u_i$ is equal to $b^m$ for some $i$. Then $u_1 \cdots u_{10}$ is a factor of $s_{n+1}=s_nb^{3^n}s_n$, so $10m < 3^{n+1}$ hence $m<3^{n-1}$. Then, by minimality of $n$, there are only two possible cases: either $u_1$ starts before the block $b^{3^n}$, or $u_1$ starts in the block $b^{3^n}$ and ends in $s_n$.

In the first case, by minimality of $n$, $u$ ends after the block $b^{3^n}$, and since no $u_i$ equals $b^m$, we get $2m>3^n$, which is in contradiction with $m<3^{n-1}$.

If $u_1$ starts in the block $b^{3^n}$ and ends in $s_n$, $u_2 \cdots u_{10}$ is a factor of $s_n=s_{n-1}b^{3^{n-1}}s_{n-1}$ and so $9m < 3^n$ hence $m<3^{n-2}$. By minimality of $n$, $u_{11}$ ends after the block $b^{3^{n-1}}$. Again, since no $u_i$ equals $b^m$, we get $2m>3^{n-1}$, which is in contradiction with $m<3^{n-2}$.

Let us then suppose that $u_{11}=b^m$, so that $u_1\cdots u_{9}$ is a factor of $s_{n+1}$. The same reasoning as before holds, since $(9m < 3^{n+1}) \Rightarrow (m < 3^{n-1})$ and $(8m < 3^n) \Rightarrow (2m < 3^{n-1})$. If $u_1=b^m$, $u_2\cdots u_{10}$ is a factor of $s_n$ with no $u_i=b^m$ and we can again apply the same reasoning.

Finally, suppose that $u_i=b^m$ with $i \neq 1$ and $i \neq 11$. Hence, $u_1 \cdots u_{10}$ is a factor of $s_{n+1}=s_nb^{3^n}s_n$, and $10m < 3^{n+1}$. If $u_1$ starts before the block $b^{3^n}$ (and $u$ ends after by minimality of $n$), we get $3m > 3^n$ since otherwise $u$ would contain two blocks $b^m$, and this contradicts $10m < 3^{n+1}$. If  $u_1$ does not start before the block $b^{3^n}$, then by minimality of $n$ it starts in this block, so $u_2\cdots u_{10}$ is a factor of $s_n=s_{n-1}b^{3^{n-1}}s_{n-1}$ which ends after the block $b^{3^{n-1}}$, again by minimality of $n$. This shows that $9m < 3^n$, and at the same time $3m>3^{n-1}$, which produces a contradiction.
\qed

\section{Paperfolding Words}

In what follows, we recall the combinatorial framework for dealing with paperfolding words introduced in~\cite{Holub}, although we use the alphabet $\{0,1\}$ instead of $\{1,-1\}$.

A paperfolding word is the sequence of ridges and valleys obtained
by unfolding a sheet of paper which has been folded infinitely many times. At each step, one can fold the paper in two different ways, thus generating
uncountably many sequences. It is known that all the paperfolding words are uniformly recurrent and have the same factor complexity $c(n)$, and that $c(n)=4n$ for $n\geq 7$ \cite{All}. 
Madill and Rampersad~\cite{MaRa} studied the abelian complexity of the regular paperfolding word and proved that  it is a $2$-regular sequence.
The regular paperfolding word \[\pap=00100110001101100010011100110110\cdots\] is the paperfolding word obtained by folding at each step in the same way. It can be defined as a Toeplitz word (see \cite{CaKa97} for a definition of Toeplitz words) as follows: Consider the infinite periodic word $\gamma=(0?1?)^{\omega}$, defined over the alphabet $\{0,1\}\cup\{?\}$. Then define $p_0=\gamma$ and, for every $n>0$, $p_n$ as the word obtained from $p_{n-1}$ by replacing the symbols $?$ with the letters of $\gamma$. So, 
\begin{align*}
 p_0 &=  0 ? 1 ? 0 ? 1 ? 0 ? 1 ? 0 ? 1 ? 0 ? 1 ? 0 ? 1 ? 0 ? 1 ?\cdots,\\
 p_1 &=  0 0 1 ? 0 1 1 ? 0 0 1 ? 0 1 1 ? 0 0 1 ? 0 1 1 ? 0 0 1 ?\cdots,\\
 p_2 &=  0 0 1 0 0 1 1 ? 0 0 1 1 0 1 1 ? 0 0 1 0 0 1 1 ? 0 0 1 1 \cdots,\\
 p_3 &=  0 0 1 0 0 1 1 0 0 0 1 1 0 1 1 ? 0 0 1 0 0 1 1 1 0 0 1 1 \cdots,
\end{align*}
 etc. Thus, $\pap=\lim_{n\to\infty}p_n$, and hence $\pap$ does not contain occurrences of the symbol $?$.

More generally, one can define a paperfolding word $\textbf{f}$ by considering the two infinite periodic words $\gamma=(0?1?)^{\omega}$ and $\bar{\gamma}=(1?0?)^{\omega}$. Then, let $\textbf{\textit{b}}=b_0b_1\cdots$ be an infinite word over $\{-1,1\}$, called {\em the sequence of instructions}.  Define $(\gamma_n)_{n\geq 0}$ where, for every $n$, $\gamma_n=\gamma$ if $b_n=1$ or $\gamma_n=\bar{\gamma}$ if $b_n=-1$. The paperfolding word $\textbf{f}$ {\em associated with} $\textbf{\textit{b}}$ is the limit of the sequence of words $f_n$ defined by  $f_0=\gamma_0$ and, for every $n>0$, $f_n$ is obtained from $f_{n-1}$ by replacing the symbols $?$ with the letters of $\gamma_n$.

Recall that every positive integer $i$ can be uniquely written as $i=2^k(2j+1)$, where 
$k$ is called the \emph{order} of $i$ (a.k.a.~the $2$-adic valuation of $i$), and $(2j+1)$ is called the \emph{odd part} of $i$. One can verify that the previous definition of $\textbf{f}$ is equivalent to the following:
for every $i=1,2,\ldots$ define
$ w_i=(-1)^jb_k$, where $i=2^k(2j+1).$
Then $f_i=0$ if $w_i=1$ and $f_i=1$ if $w_i=-1$. This is equivalent to 
\[f_i=1 \quad \text{iff} \quad i \equiv 2^k(2+b_k) \mod{2^{k+2}}.\]

\begin{remark}
The regular paperfolding word corresponds to the sequence of instructions $\textbf{\textit{b}}=1^{\omega}$.
\end{remark}

\begin{definition}
Let $\textbf{f}$ be a paperfolding word. An occurrence of a letter in $\textbf{f}$ at position $i$ is said to be \emph{of order $k$} if the letter at position $i$ is $?$ in $f_{k-1}$ and different from $?$  in $f_{k}$. We consider the letters occurring in $f_0$ as of order $0$. 
\end{definition}

Hence, in a paperfolding word $\textbf{f}$ associated with the sequence $\textbf{\textit{b}}=b_0b_1\cdots$, the $1$'s of order $0$ appear at positions $2+b_0+4t$, $t\geq 0$, the $1$'s of order $1$ appear at positions $2(2+b_1+4t)$, $t\geq 0$, and, in general, the $1$'s of order $k$ appear at positions $2^k(2+b_k+4t)$, $t\geq 0$.

Let $\textbf{f}=f_1f_2\cdots$ be a paperfolding word associated with the sequence $\textbf{\textit{b}}=b_0b_1\cdots$. A factor of $\textbf{f}$ of length $n$ starting at position $\ell+1$, denoted by $\textbf{f}[\ell+1,\ldots,\ell+n]$, contains a number of  $1$'s that is given by the sum, for all $k\geq 0$, of the  $1$'s of order $k$ in the interval $[\ell+1,\ell+n]$. For each $k$, since the  $1$'s of order $k$ are at distance $2^{k+2}$ one from another, the number of occurrences of $1$'s of order $k$ in $\textbf{f}[\ell+1,\ldots,\ell+n]$ is given by
\[\left\lfloor  \dfrac{n-\ell}{2^{k+2}} \right\rfloor + \epsilon_{k,b_k}(\ell,n),\]
where $\epsilon_{k,b_k}(\ell,n)\in\{0,1\}$ depends on the sequence $\textbf{\textit{b}}$ (in fact, $b_k$ determines the positions of the occurrences of the $1$'s of order $k$ in $\textbf{f}$). 
We set \[\Delta(\ell,n)=\sum_{k\geq 0}\epsilon_{k,b_k}(\ell,n)\] the number of ``extra'' $1$'s in $\textbf{f}[\ell+1,\ldots,\ell+n]$. 

For example, in the prefix $\pap[1,14]$ of length $14$ of the regular paperfolding word, we know that there are  at least $3=\lfloor\frac{14}{4}\rfloor$ $1$'s of order $0$, $1=\lfloor\frac{14}{8}\rfloor$ of order $1$ and $0=\lfloor\frac{14}{16}\rfloor$ of order $2$. In the interval $[1,14]$ there are three $1$'s of order $0$ (at positions $3$, $7$ and $11$), two $1$'s of order $1$ (at positions $6$ and $14$), and one $1$ of order $2$ (at position $12$), so  we have in $\pap[1,14]$ no extra $1$ of order $0$, i.e., $\epsilon_{0,1}(0,14)=0$, one extra $1$ of order $1$, i.e., $\epsilon_{1,1}(0,14)=1$ and one extra $1$ of order $2$, i.e., $\epsilon_{2,1}(0,14)=1$, so that $\Delta(0,14)=2$. 

We set
\[\mathcal{E}_{k,b_k}(\ell,d,m)=\left(\epsilon_{k,b_k}(\ell,\ell+d),\ldots,\epsilon_{k,b_k}(\ell+(m-1)d,\ell+md)\right)\]
and 
\[\Delta(\ell,d,m)=\sum_{k\geq 0}\mathcal{E}_{k,b_k}(\ell,d,m)=\left(\Delta(\ell,\ell+d),\ldots,\Delta(\ell+(m-1)d,\ell+md)\right).\]

The factor of $\textbf{f}$ of length $dm$ starting at position $\ell+1$ is an abelian $m$-power if and only if the components of the vector $\Delta(\ell,d,m)$ are all equal, while it is an abelian $m$-antipower if and only if the components of the vector $\Delta(\ell,d,m)$ are pairwise distinct.

The next result (Lemma 4 of \cite{Holub}) will be the fundamental ingredient for the construction of abelian antipowers in paperfolding words.

\begin{lemma}[Additivity Lemma]\label{al}
Let $\ell,\ell'\geq 0$ and $m,d,d'\geq 1$ be integers with $\ell'$
 and $d'$ both even. Let $r$ be such that $2^r>\ell+md$, and for each $k\geq 0$ the following implication holds: if $\mathcal{E}_{k,1}(\ell',d',m) \neq \mathcal{E}_{k,-1}(\ell',d',m)$ then $b_k=b_{k+r}$.

Then \[\Delta(\ell,d,m) + \Delta(\ell',d',m)=\Delta(\ell + 2^r\ell', d+2^rd',m).\]   
 \end{lemma}

Using the Additivity Lemma, Holub \cite{Holub} proved that all paperfolding words contain abelian powers of every order. We will use the Additivity Lemma to prove that all paperfolding words contain abelian antipowers of every order. We start with the regular paperfolding word, then we extend the argument to all paperfolding words.

\subsection{Regular paperfolding word}

Let 
\[\begin{array}{ccccc}
\Phi & : & \lbrace0,1\rbrace^2 & \to & \lbrace x,y,z \rbrace \\
 & & 00 & \mapsto & x \\
 & & 01 & \mapsto & y \\
 & & 10 & \mapsto & y \\
 & & 11 & \mapsto & z \\
\end{array}\]
be the morphism that identifies words of length $2$ over the alphabet $\{0,1\}$ that are abelian equivalent. We have the following lemma:

\begin{lemma}\label{lem:1}
Let $n \geq 3$ be an integer. Let $p=\pap[\ell+1,\ldots,\ell+2^n]=u_1v_1\cdots u_{2^{n-1}}v_{2^{n-1}}$ be a factor of $\pap$ of length $2^n$.
Then, no $q<2^{n-1}$ exists such that 

\begin{align}\label{eq:1}
\Phi(p)=
\Phi(u_1v_1)\cdots\Phi(u_{2^{n-1}}v_{2^{n-1}})=\Phi(u_{q+1}v_{q+1})\cdots\Phi(u_{2^{n-1}}v_{2^{n-1}})\Phi(u_1v_1)\cdots\Phi(u_qv_q).
\end{align}
\end{lemma}

\begin{proof}
First, notice that if $q'$ is the smallest solution of (\ref{eq:1}), then $q'|2^{n-1}$. Indeed, writing $w_i=\Phi(u_iv_i)$, we have 
\begin{align*}
 w_1\cdots w_{2^{n-1}}& =  w_1\cdots w_{q'}w_{q'+1}\cdots w_{2^{n-1}} \\
 & =w_{q'+1}\cdots w_{2^{n-1}}w_1\cdots w_{q'},
\end{align*}
and since two words commute if and only if they are powers of the same word, there exists a word $z$ and positive integers $s$ and $t$ such that
\begin{align*}
w_1\cdots w_{q'}=z^s \text{ and } w_{q'+1}\cdots w_{2^{n-1}}=z^t.
\end{align*}
This gives $|z|\cdot(s+t)=2^{n-1}$ and $|z|\cdot s=q'$.
By the minimality of $q'$, we have that $s=1$ and so $|z|=q'$ divides $2^{n-1}$. Thus, $q'=2^j$ for some integer $j <n$.

By the Toeplitz construction of $\pap$, we immediately have that 
\[u_1v_1\cdots u_{2^{n-1}}v_{2^{n-1}}=av_1\overline{a}v_2av_3\overline{a}\cdots \overline{a} v_{2^{n-1}} \]
or
\[u_1v_1\cdots u_{2^{n-1}}v_{2^{n-1}}=u_1au_2\overline{a} u_3au_4\overline{a}\cdots u_{2^{n-1}}\overline{a} \]
with $a\in \lbrace 0,1 \rbrace $ and $\overline{a}=1-a$. 

Suppose $q' \neq 1$ and $q'\neq 2^{n-1}$. Since $q'$ is even, we have that $\Phi(u_iv_i)=\Phi(u_{i+q'}v_{i+q'})$ implies $u_iv_i=u_{i+q'}v_{i+q'}$.
But this cannot be the case, since two consecutive letters of order $j$ occur in $\pap$ at distance $2^{j+1}$. Since $j \leq n-2$, we have $2^{j+2} \leq 2^n$, so the factor $p$ contains at least two consecutive letters of order $j$. Suppose that the first of such letters is $u_i$; then $u_{i+q'}$ is at distance $2q'=2^{j+1}$, so $u_{i+q'} \neq u_{i}$, against the hypothesis that $q'$ is a solution of (\ref{eq:1}).

Thus, we must have $q'=1$ or $q'=2^{n-1}$. Since $n \geq 3$, $\pap[\ell+1,\ldots,\ell+2^n]$ contains two consecutive letters of order $1$. Let us first suppose that $v_i$ is a $1$ of order $1$, $u_i$ is a 1 of order 0 and $v_{i+2}$ is a $0$ of order $1$. Then, 
$\Phi(u_i v_i)=\Phi(11) \neq \Phi(10)=\Phi(u_{i+2}v_{i+2})$. The other cases would give $10 u_{i+1} v_{i+1}11$ with $v_i$ a $0$ of order 1 and $v_{i+2}$ a $1$ of order 1, $00 u_{i+1} v_{i+1}01$ and $00 u_{i+1} v_{i+1}01$ respectively in the case $u_i$ is a 0 of order 0. Similary, we get $10u_{i+1} v_{i+1} 00$ and $00 u_{i+1} v_{i+1} 10$ if $u_i$ is a $1$ of order 1 and $u_{i+2}$ a $0$ of order 1 or vice versa, and $v_i$ a 0 of order 0. The cases with $v_i$ a 1 of order 0 are symetric. Every case leads to $\Phi(u_i v_i) \neq \Phi(u_{i+2}v_{i+2})$.
This implies $q'\neq 1$ and so $q'=2^{n-1}$. By minimality of $q'$, the only solution of (\ref{eq:1}) is $q=2^{n-1}$.
\end{proof}

\begin{theorem}
The regular paperfolding word contains abelian $m$-antipowers for every $m\geq 2$.
\end{theorem}

\begin{proof}
The proof is mainly based on the Additivity Lemma. Let $m\geq 2$ be fixed. To prove the result it is sufficient to find a vector $\Delta(s,d,m)$ having pairwise distinct components. Let $k$ be an integer  such that $2^k\geq m$. Consider the first factor of length $2^{k+2}-1$ containing a $1$ of order $k$ in the middle; our factor is then of the form
\[ w1w'
\]
with $|w|=|w'|=2^{k+1}-1$. Since for every positive integers $i,k',s$, we have $ p_i \text{ of order } k' \Rightarrow  p_{i+2^{k'+s}} \text{ of order } k'$ and
\begin{equation*}
p_i \text{ of order } k' \Rightarrow  p_{i+2^{k'+2}}=p_i\neq p_{i+2^{k'+1}}
\end{equation*}
we get:
\begin{equation}\label{eq:2}
p_i \text{ of order } k' \Rightarrow  p_{i+2^{k'+2+s}}=p_i\neq p_{i+2^{k'+1}}
\end{equation} 

then, up to applying a translation, we can suppose $w=w'$. In fact, since $|w1|=2^{k+1}$, the equality is true for every letter of order smaller than $k$ by (\ref{eq:2}). Now, take the smallest order $r>k$ of a letter $0$ in $w$ or $w'$. It is the only letter of this order in our factor since two letters of order $r$ are distant of $2^{r+1}>|w1w'|$. If we consider the factor translated by $2^{r+1}$, by (\ref{eq:2}) the letters of order smaller than $r$ are the same and the letter we considered becomes a $1$. Since the length of $w1w'$ is $2^{k+2}-1$ and the distance between two letters of order higher than $k$ is at least $2^{k+1}$, the factor $w1w'$ contains exactly two letters of order higher than $k$. Hence, in at most $2$ steps we get $w1w$ with every letter of order greater than $k$ being a $1$. Writing $\ell+1$ the starting position  of an occurrence in $\pap$ of the factor $w1w$, we set $\ell'=\ell$ if $\ell$ is even or $\ell'=\ell+1$ otherwise. Consider the vectors \[\Delta(\ell',2,2^k), \Delta(\ell'+2,2,2^k), \Delta(\ell'+4,2,2^k), \Delta(\ell'+6,2,2^k), \ldots , \Delta(\ell'+2^{k+1}-2,2,2^k).\]
We claim that these vectors are pairwise distinct. By contradiction, if $\Delta(\ell'+2p,2,2^k)=\Delta(\ell'+2q,2,2^k)$ for some $p,q$ with $p \leq q$, then we  have that 
\begin{equation}\label{eq:3}
\Phi(p_{\ell'+2p+1}\cdots p_{\ell'+2p+2^{k+1}})=\Phi(p_{\ell'+2q+1}\cdots p_{\ell'+2q+2^{k+1}}).
\end{equation}
 Since the factor we are considering is $w1w$, we have $p_{\ell'+2p+1}\cdots p_{\ell'+2q}=p_{\ell'+2p+1+2^{k+1}}\cdots p_{\ell'+2q+2^{k+1}}$ and so 
\[\Phi(p_{\ell'+2q+1}\cdots p_{\ell'+2q+2^{k+1}})=\Phi(p_{\ell'+2q+1}\cdots p_{\ell'+2p+2^{k+1}}p_{\ell'+2p+1}\cdots p_{\ell'+2q})\]
but this and (\ref{eq:3}) contradicts Lemma~\ref{lem:1}. 

Finally, as the vectors are different, we use the Additivity Lemma to obtain a vector whose components are pairwise distinct: applying $n$ times the Additivity Lemma on $\Delta(\ell'+2p,2,2^k)$ one can obtain $n\Delta(\ell'+2p,2,2^k)$. It then suffices to take a  sequence of integers $\alpha_0,\ldots,\alpha_{2^{k}-1}$ increasing enough to have 
\[\Sigma_{i=0}^{2^{k}-1}\alpha_i\Delta(s'+2i,2,2^k), \]
a vector whose components are pairwise distinct. Indeed, labelling $a_j$ the $j$-th component of this vector and $x_{i,j}$ the $j$-th component of $\Delta(s'+2i,2,2^k)$, we have
\[a_j=a_{j'} \Leftrightarrow \Sigma_{i=0}^{2^{k}-1}\alpha_ix_{i,j}=\Sigma_{i=0}^{2^{k}-1}\alpha_ix_{i,j'} \Leftrightarrow \Sigma_{i=0}^{2^{k}-1}\alpha_i(x_{i,j}-x_{i,j'})=0.\]
By ``increasing enough'', we precisely mean $\alpha_r > \Sigma_{i=0}^{r-1}\alpha_i\underset{0\leq q,q' \leq 2^k-1}{\sup}(x_{i,q}-x_{i,q'})$, so that by decreasing induction we have that for every $i$, with $ 0\leq i \leq 2^k-1$, one has $x_{i,j}=x_{i,j'}$.
In particular, this  gives $\Delta(\ell'+2j,2,2^k)=\Delta(\ell'+2j',2,2^k)$, which implies $j=j'$. Hence, all the components are pairwise distinct and the proof is complete.
\end{proof}

\subsection{All paperfolding words}

To generalize the result above to all paperfolding words, one has to take care of the condition  $b_i=b_{i+r}$ in the Additivity Lemma.

Lemma~\ref{lem:1} can be modified so that the translation is not by $2$ but by $2^u$, for any $u> 1$.
Let 
\[\begin{array}{ccccc}
\phi & : & \lbrace0,1\rbrace^{2^u} & \to & \mathbb{N} \\
 & & a_1\cdots a_{2^u} & \mapsto & |\lbrace i \mid a_i=1 \rbrace|
\end{array}\]
be the morphism that identifies words of length $2^u$ over $\{0,1\}$ that are abelian equivalent. Then we have the following lemma, analogous to Lemma~\ref{lem:1}:

\begin{lemma}\label{lem:2}
Let $n \geq u+3$ be an integer and let $\textbf{f}$ be a paperfolding word. Every factor $f=\textbf{f}\left[ \ell+1,\ell+2^n\right]=a_{1,1}a_{1,2}\cdots a_{2^{n-1},2^u-1}a_{2^{n-1},2^u}$ of $\textbf{f}$ of length $2^n$ satisfies the following property:
If $q$ is such that
\begin{align*}
&\phi(f)=\phi(a_{1,1}\cdots a_{1,2^u})\cdots \phi(a_{2^{n-1},1}\cdots a_{2^{n-1},2^u})=\\ 
&\phi(a_{q+1,1}\cdots a_{q+1,2^u})\cdots \phi(a_{2^{n-1},1}\cdots a_{2^{n-1},2^u})\phi(a_{1,1}\cdots a_{1,2^u})\cdots \phi(a_{q,1}\cdots a_{q,2^u}),
\end{align*}
then $q=2^{n-1}$.
\end{lemma}

\begin{proof}
The proof of Lemma~\ref{lem:1} mainly applies here; we only need to change the part where we use the Toeplitz construction to justify $j=n-1$. Here, in each $2^u$-tuple one can find one letter of order $u-1$ and one letter of higher order. Using (\ref{eq:2}), we then see that $\phi(a_{i,1}\cdots a_{i,2^u})$ is totally determined by the letter of order $u-1$ and the letter of higher order in $a_{i,1}\cdots a_{i,2^u}$. Applying again (\ref{eq:2}) to the letter of order $u-1$, we can apply exactly the same reasoning as in the proof of Lemma~\ref{lem:1} (in a sense, our new $\phi$ is the previous one modulo the letters of order smaller than $u-1$). 
\end{proof}

Now, we can prove the main theorem:

\begin{theorem}\label{thm:main}
Every paperfolding word \textbf{f} contains abelian $m$-antipowers for every $m\geq 2$.
\end{theorem}

\begin{proof}
Let $k$ be an integer such that $2^k \geq m$. As before, we will prove that \textbf{f} contains abelian $2^k$-antipowers, hence it will contain abelian $m$-antipowers.
Since the alphabet $\lbrace 0,1\rbrace$ is finite, there must exist a factor $b_{u-1}\cdots b_{u+k+4}$ of \textbf{\textit{b}} that occurs infinitely often. As before, let us start with the first block of length $2^{u+k+2}-1$ containing a $1$ of order $u+k$ in the middle; our block is then 
\[ w1w'
\]
with $|w|=|w'|=2^{u+k+1}-1$.
As before, in at most two steps, we can have $w=w'$, and the maximum order of a letter appearing in this factor is $u+k+4$. Again, writing $\ell$ the starting position of an occurrence of this factor, we set $\ell'=\ell$ if $\ell$ is even or $\ell'=\ell+1$ otherwise. Consider the vectors 
\[\Delta(\ell',2^u,2^k), \Delta(\ell'+2^u,2^u,2^k), \Delta(\ell'+2^{u+1},2^u,2^k), \ldots ,\Delta(\ell'+2^{u+k+1}-2^u,2^u,2^k).\] 
Here again, these vectors are pairwise distinct: if $\Delta(\ell'+2^up,2^u,2^k)=\Delta(\ell'+2^uq,2^u,2^k)$, we  have that  
\[\phi(p_{\ell'+2^up+1}\cdots p_{\ell'+2^u(p+2^{k})})=\phi(p_{\ell'+2^uq+1}\cdots p_{\ell'+2^u(q+2^{k})}) \] 
and this contradicts Lemma~\ref{lem:2} because, here again, $w=w'$ and so
\[p_{\ell'+2^up+1}\cdots p_{\ell'+2^uq}=p_{\ell'+2^u(p+2^{k})+1}\cdots p_{\ell'+2^u(q+2^{k})}.\]

Moreover, $\varepsilon_{i,-1}( \ell'+2^up,2^u,2^k) \neq \varepsilon_{i,1}( \ell'+2^up,2^u,2^k) \Rightarrow u-1 \leq i \leq u+k+4$, using (\ref{eq:2}) and the fact that no letter of order higher than $u+k+4$ appears in the factor $w1w$. So, choosing $r$ such that $2^r > \ell'+2^{u+k+1}-2^u + 2^{u+k}$ and $b_{u-1}\cdots b_{u+k+4}=b_{r+u-1}\cdots b_{r+u+k+4}$, we can apply the Additivity Lemma and, as for the regular paperfolding word, construct an abelian $2^k$-antipower that occurs as a factor in \textbf{f}.
\end{proof}

\begin{remark}\label{rem:pap}
 From Theorem~\ref{thm:main} it follows immediately that every paperfolding word has unbounded abelian complexity.
\end{remark}

In~\cite{CRSZ2010} Cassaigne et al.~proved that every infinite word $w$ with bounded abelian complexity $a_w(n)$ contains abelian powers of every order. In fact, one can see that the following hypothesis on $w$ is sufficient:
\begin{equation}\label{eq:C}
 \exists N, \forall m, \exists v \in \text{Fact}(w), |v|=m  \text{ and } a_v(n) \leq N,
\end{equation}
that is, the abelian complexity is bounded on arbitrarily long factors of $w$.
Since every paperfolding word is uniformly recurrent, by Remark \ref{rem:pap} we have that (\ref{eq:C}) cannot hold true for paperfolding words. Hence, (\ref{eq:C}) is not a necessary condition to have abelian powers of every order.


\begin{thebibliography}{10}

\bibitem{MoHe}
Morse, M., Hedlund, G.A.:
\newblock Symbolic dynamics.
\newblock Amer. J. Math. \textbf{60} (1938)  1--42

\bibitem{minsub}
Richomme, G., Saari, K., Zamboni, L.:
\newblock Abelian complexity of minimal subshifts.
\newblock J. Lond. Math. Soc. \textbf{83}(1) (2011)  79--95

\bibitem{Holub}
\v{S}tepan Holub:
\newblock Abelian powers in paper-folding words.
\newblock J. Comb. Theory, Ser. {A} \textbf{120}(4) (2013)  872--881

\bibitem{icalp}
Fici, G., Restivo, A., Silva, M., Zamboni, L.Q.:
\newblock Anti-powers in infinite words.
\newblock J. Comb. Theory, Ser. {A} \textbf{157} (2018)  109--119

\bibitem{BFR}
Blanchet-Sadri, F., Fox, N., Rampersad, N.:
\newblock On the asymptotic abelian complexity of morphic words.
\newblock Advances in Applied Mathematics \textbf{61} (2014)  46--84

\bibitem{FL11}
Fici, G., Lipt{\'a}k, {\relax Zs}.:
\newblock On prefix normal words.
\newblock In: Proc.\ of the 15th Intern.\ Conf.\ on Developments in Language
  Theory (DLT 2011). Volume 6795 of LNCS., Springer (2011)  228--238

\bibitem{PN17}
Burcsi, P., Fici, G., Lipt{\'{a}}k, Z., Ruskey, F., Sawada, J.:
\newblock On prefix normal words and prefix normal forms.
\newblock Theoret. Comput. Sci. \textbf{659} (2017)  1--13

\bibitem{All}
Allouche, J.:
\newblock The number of factors in a paperfolding sequence.
\newblock Bull. Austral. Math. Soc. \textbf{46} (1992)  23--32

\bibitem{MaRa}
Madill, B., Rampersad, N.:
\newblock The abelian complexity of the paperfolding word.
\newblock Discr. Math. \textbf{313}(7) (2013)  831--838

\bibitem{CaKa97}
Cassaigne, J., Karhum{\"{a}}ki, J.:
\newblock Toeplitz words, generalized periodicity and periodically iterated
  morphisms.
\newblock Eur. J. Comb. \textbf{18}(5) (1997)  497--510

\bibitem{CRSZ2010}
Cassaigne, J., Richomme, G., Saari, K., Zamboni, L.:
\newblock Avoiding {A}belian powers in binary words with bounded {A}belian
  complexity.
\newblock \textbf{22}(4) (2011)  905--920

\end{thebibliography}

\end{document}